\documentclass[12pt]{amsart}
\usepackage[utf8]{inputenc}
\usepackage[russian,english]{babel}
\usepackage{amscd,amssymb, latexsym}
\usepackage{amsmath,amsthm,amstext, amsbsy, amssymb,latexsym}
\usepackage{epsfig,graphics,graphicx,mathrsfs,color}
\usepackage{a4wide}
\usepackage{ulem}
\newtheorem{theorem}{Theorem}
\newtheorem{lemma}{Lemma}

\newtheorem{definition}{Definition}

\theoremstyle{remark}

\makeatletter
\@addtoreset{equation}{section}
\makeatother

\def\Re{\mathop{\rm Re}\nolimits}
\def\Im{\mathop{\rm Im}\nolimits}
\def\id{\mathop{\rm id}\nolimits}
\begin{document}

\title[On $L$-special domains...]{On $L$-special domains with algebraic boundaries}

\author[M.~Borovikov]{Mikhail Borovikov}

\begin{abstract}
The concept of $L$-special domain appeared in the early 2000s. This
analytical characteristic of domains in the complex plane is related to the
problem on uniform approximation of functions on Carath\'eodory compacts in
$\mathbb{R}^2$ by polynomial solutions of homogeneous second-order elliptic
partial differential equations $Lu=0$ with constant complex coefficients. In
this paper, new properties and examples of $L$-special domains with algebraic
boundaries are obtained.
\end{abstract}

\thanks{The study was carried out with the financial support of the Ministry of Science and Higher Education of the Russian Federation in the framework of a scientific project under agreement No. 075-15-2025-013. 
Author is a member of the group “Leader” followed by Alexander Bufetov
which won the contest conducted by
Foundation for the Advancement of Theoretical Physics and Mathematics “BASIS”
and would like to thank its sponsors and jury.}

\address{
\hskip -\parindent Mikhail Borovikov${}^{1,2,3}$:
\newline \indent 1)~Faculty of Mechanics and Mathematics,
\newline \indent Lomonosov Moscow State University, Moscow, Russia;
\newline \indent 2)~Moscow Center for Fundamental and Applied Mathematics,
\newline \indent Lomonosov Moscow State University, Moscow, Russia;
\newline \indent 3)~St.~Petersburg State University, St.~Petersburg, Russia;
\newline \indent {\tt misha.borovikov@gmail.com}
}

\maketitle

\section{Introduction}

Let $L$ be a second-order elliptic partial differential operator in
$\mathbb{R}^2$ with constant complex coefficients, i.e.,
$Lf=af_{xx}+bf_{xy}+cf_{yy}$, $a,b,c\in\mathbb C$. The problem on uniform
approximability of functions on compact sets in $\mathbb R^2$ by $L$-analytic
polynomials (that is by polynomials satisfying the equation $Lf=0$)
attracted the interest of analysts since early 1990s. The necessary and
sufficient approximability conditions in this problem were obtained for
Carath\'eodory compact sets in terms of a special analytical characteristic
of bounded simply connected domains, which is expressed by the property of
the domain to be $L$-special.

Let us recall the corresponding definitions. A bounded domain $D\subset\mathbb R^2$ is
called a Carath\'eodory domain, if $\partial D=\partial D_\infty$, where
$D_{\infty}$ is the unbounded (connected) component of
$\mathbb{R}^2\setminus\overline{D}$. It can be readily verified that any
Carath\'eodory domain is simply connected and coincides with the interior of
its closure. Let $\varphi$ be a conformal mapping of the disc
$\mathbb{D}=\{z:|z|<1\}$ onto $D$. One says that a holomorphic function $f$
in $D$ belongs to the class $AC(D)$ if the function $f\circ\varphi$ can be
extended to a function that is continuous on $\mathbb{D}$ and absolutely
continuous on the unit circle. It is known that for every $f\in AC(D)$, for
every accessible boundary point $\zeta$ of $D$, and for every path
$\gamma$ lying in $D\cup\{\zeta\}$ and ending at $\zeta$, the limit of $f$
along $\gamma$ exists and is equal to the same value $f(\zeta)$, which is
called a boundary value of $f$ at $\zeta$.

Let $\lambda_1$ and $\lambda_2$ be the roots of the characteristic polynomial
$a\lambda^2+b\lambda+c=0$ for $L$. The (complex) numbers $\lambda_1$ and
$\lambda_2$ are not real due to ellipticity of $L$. We associate with $L$ the
non-degenerate real-linear transformations $T_k$, $k=1,2$ of the plane:
\begin{equation*}
T_k : z=x+iy \rightarrow x+\frac{1}{\lambda_k}y,\quad k=1,2.
\end{equation*}

\begin{definition}\label{def1}
A Carath\'eodory domain $D$ is called $L$-special, if there exist two
non-constant functions $F_1\in AC(T_1D)$ and $F_2\in AC(T_2D)$ such that
$F_1(T_1\zeta)=F_2(T_2\zeta)$ for every accessible boundary point
$\zeta\in\partial D$.
\end{definition}

If $D$ is a $L$-special domain, then the pair of functions $(F_1,F_2)$ taken
from Definition~\ref{def1} is called \textit{admissible} for $D$. Notice that, for a given $L$-special domain
$D$, the admissible pair is not uniquely determined.

Recall that a compact set $K\subset\mathbb R^2$ is called a Carath\'eodory
compact set, if $\partial K=\partial \widehat{K}$, where $\widehat{K}$ is the
union of $K$ and all bounded connected components of the set
$\mathbb{R}^2\setminus K$. In \cite{Z02,Z03,Z04,BFM24}
necessary and sufficient conditions on a Carath\'eodry compact set $K$ were
obtained in order that every function $f$ continuous on $K$ and satisfying the
equations $Lf=0$ on the interior $K^\circ$ of $K$ can be approximated
uniformly on $K$ with an arbitrary accuracy by $L$-analytic polynomials.
These conditions are formulated in terms of $L$-special domains. In the
papers cited above several conditions established in order that a given
domain $D$ is not $L$-special for certain operators $L$ of the type under
consideration; however, the concept of $L$-speciality itself is still not
well studied. In particular, no description of $L$-special domains is known
in terms of properties of conformal or univalent harmonic mappings of the
disc (onto the domain under consideration), and, moreover, only a few
explicit examples of such domains are discovered up to now. Our aim is to
present new construction of $L$-special domain with an algebraic boundary which
answers, in particular, the question posed in the early 2000s about the
existence of such domains different from ellipses. We also note that the
question of uniform approximability by $L$-analytic polynomials is closely
related to the questions of uniqueness and existence of solution to the
Dirichlet problem for the equation $Lf=0$ in bounded simply connected domains
in the plane; more details about these questions can be found in the 
\cite{B23,BF22,BFM24}.

\section{Main Results}

Let $\mathbb{C}[x,y]$ and $\mathbb{R}[x,y]$ denote the spaces of polynomials
in two variables with complex and real coefficients, respectively. Let
$\varGamma$ in $\mathbb{R}^2$ be a Jordan curve possessing the property
$\varGamma\subset\{(x,y) : P(x,y)=0\}$ for some $P\in\mathbb{C}[x,y]$. In
this case we say that $\varGamma$ is \textit{algebraic} and $P$
\textit{defines} $\varGamma$. The \textit{order} of an algebraic curve
$\varGamma$ is the smallest number $n$ such that there exists a polynomial of
degree $n$ defining $\varGamma$.

A Jordan domain is said to be a domain with an algebraic boundary, if it is the
interior of some closed Jordan algebraic curve. If a given Jordan domain $D$
is $L$-special, and if there is an admissible pair for $D$ consisting of
polynomials, then $D$ is obviously a domain with an algebraic boundary.

Recall that the elliptic operator $L$ under consideration is strongly
elliptic, if its characteristic roots $\lambda_1$ and $\lambda_2$ introduced
above have the opposite signs of imaginary parts. Since $L$-special domains
do not exist for strongly elliptic operators (see, for example,
\cite[Corollary~1]{Z02}), in what follows we deal with only not-strongly
elliptic $L$. For non-strongly elliptic operators $L$, there exists a
non-degenerate linear transformation of $\mathbb{R}^2$ that reduces $L$ to
the form
$$
L_\beta=c\overline{\partial}\partial_\beta
$$
where $c\in\mathbb{C}$, $c\neq0$, while
$\overline{\partial}=\frac12(\partial/\partial x+i\partial/\partial y)$ is
the usual Cauchy-Riemann operator and
$$
\partial_\beta:=\frac12\Bigr(\frac{\partial}{\partial x}+i\beta\frac{\partial}{\partial y}\Bigl)
$$
for some $\beta=\beta(L)\in(0,1)$.

Notice that the condition $0<\beta<1$ singles out precisely the class of
non-strongly elliptic operators, while every strongly elliptic $L$ under
consideration can be reduced to the operator of the form $L_\beta$ with
$\beta\in(-1,0)$.

Further we will assume that the operator $L$ under consideration already
has the form $L_\beta$ with $\beta\in(0,1)$. In this case, the
transformations of the plane $T_1$ and $T_2$ used in Definition~\ref{def1}
are $T_1=\id$ and $T_2=T_\beta : z\mapsto z_\beta=x+\frac{i}{\beta}y$.

Let us present a simple example of a $L$-special domain, which first appeared
in \cite{Z02}: this is the interior of the ellipse
$$
\Big\{(x_1,x_2) : x_1^2+\frac1\beta \,x_2^2=1\Big\}
$$
and one of the admissible pairs for this domain is $(F_1,F_2)$ where
$$F_1(z)=\frac{z^2} {1-\beta},\quad 
F_2(z_\beta)=1-\frac{\beta z_\beta^2}{1-\beta}\,.
$$
Indeed, putting $x+iy$ and $x+\frac{i}{\beta}y$ in places of $z$ and
$z_\beta$ into the equation $F_1(z)=F_2(z_\beta)$, we obtain exactly the
equation of the ellipse under consideration.

Let $n>2$ be an integer. The first result of the present paper states that
the equation $F_1(z)=F_2(z_\beta)$, where $F_1$ and $F_2$ are polynomials
with $\max(\deg(F_1),\deg(F_2))\leqslant n$, cannot define any algebraic
curve of order $n$. Thus, a domain with an algebraic boundary of order $n>2$ can
be $L$-special with an admissible pair of polynomials $(F_1,F_2)$ only in the
case when the maximal degree of $F_1$ and $F_2$ is greater than the order of
its boundary. 

\begin{theorem}\label{th2}
Let $L=L_\beta$ with $\beta\in(0,1)$. Let $D$ be a domain with an algebraic
boundary of order $n>2$ such that $D$ is $L$-special and suppose the pair of
polynomials $(F_1,F_2)$ to be an admissible pair for~$D$. Then $\max(\deg
F_1,\deg F_2)>n$.
\end{theorem}

For proving Theorem~\ref{th2} we need the following two (most likely,
commonly known) lemmas. The first one is as follows.
\begin{lemma}\label{l1}
Let $D$ be a domain  with an algebraic
boundary $\varGamma$ and a polynomial $P\in\mathbb{C}[x,y]$, $\deg P=m$, define $\varGamma$. Then there exists $R\in\mathbb{R}[x,y]$, $\deg R=n$, such that $R$
also defines $\varGamma$ and $P$ is divisible by $R$, where $n$ is the order
of the boundary of $D$. In particular, if $m=n$, then $P=\gamma R$, where
$\gamma$ is a complex number.
\end{lemma}

Let $n>2$ and $P\in\mathbb{C}[x,y]$. Since
\begin{equation*}
x=\frac{z-\beta z_\beta}{1-\beta},\quad y=i\frac{\beta z- \beta z_\beta}{1-\beta},
\end{equation*}
the coefficients of the polynomial
$Q(z,z_\beta)=P\bigl(x(z,z_\beta),y(z,z_\beta)\bigr)\in\mathbb{C}[z,z_\beta]$
depend linearly on the coefficients of $P$. Moreover, the coefficients of $Q$
at monomials of degree $k$ depend only on the coefficients of $P$ at
monomials of degree $k$. Thus, we can define a linear operator $S_\beta$
acting in the algebra $\mathbb{C}[x,y]$ that maps the polynomial $P$ to the
polynomial $Q$ according to the rule described above. If we substitute the
expressions for $z$ and $z_\beta$ in terms of $x$ and $y$ into some
polynomial $\widetilde{Q}(z,z_\beta)$, then we can define by the same way the
operator $\widetilde{S_\beta}$, which maps the polynomial $\widetilde{Q}$ to
$\widetilde{P}(x,y)=\widetilde{Q}\bigl(z(x,y),z_\beta(x,y)\bigr)$. Using this
notation we can state the second aforementioned lemma.
\begin{lemma}\label{l2}
The operator $S_\beta$ is an automorphism of $\mathbb{C}[x,y]$ and
$\widetilde{S_\beta}$ is the inverse to $S_\beta$.
\end{lemma}

The result of Theorem \ref{th2} can be extended for polynomials whose degree
is greater than the order of the curve and the following statement holds.
\begin{theorem}\label{th3}
Let $n>2$ and $k$, $1\leqslant k\leqslant n$ be integers such that $n$ is not
divisible by $k$. Let $D$ be a domain with an algebraic boundary such that the
order of $\varGamma=\partial D$ is $n$ and $\varGamma$ is defined by the
equation $P(x,y)=c$, where $P$ is a homogeneous polynomial of degree $n$. If
$D$ is $L$-special and if the pair of polynomials $(F_1,F_2)$ is an
admissible pair for $D$, then $\max(\deg F_1,\deg F_2)\neq n+k$.
\end{theorem}

A special type of homogeneous polynomials appears in the proof of this
theorem.
\begin{definition}
A homogeneous polynomial $F(x,y)$ is said to be diagonal, if it has the
form $ax^n+by^n$ for some $a,b\in\mathbb{C}$.
\end{definition}
We need the following simple lemma concerning diagonal polynomials.
\begin{lemma}\label{l3}
If a non-zero diagonal polynomial $P$ of degree $n$ is divisible by a
diagonal polynomial $Q$ of degree $k$, then $n$ is divisible by $k$.
\end{lemma}
The proof of this lemma follows directly from the fact that the group of
roots of the polynomial $z^k-1$ is a subgroup of the roots of the polynomial
$z^n-1$.

Now we will describe the main result of the paper, it gives a new example of $L$-special
domain with an algebraic boundary.

\begin{theorem}\label{th4}
There exists an elliptic operator $L$ of the form $L_\beta$ with
$\beta\in(0,1)$, and a domain $D$ with the algebraic boundary $\varGamma$ of
the order $4$, such that $D$ is $L$-special with an admissible pair
consisting of polynomials.
\end{theorem}

The idea of the proof of Theorem~\ref{th4} is the following: we take the
domain bounded by the fourth degree algebraic curve, defined by the
polynomial $P(x,y)-1$, where
\begin{equation}\label{pol}
P(x,y)=x^4+\Bigl(\frac{2}{\beta}-4\alpha^2\Bigr)x^2y^2+\frac{1}{\beta^2}\,y^4
\end{equation}
for some $0<\alpha<1/\sqrt{\beta}$. The condition $\alpha<1/\sqrt{\beta}$
guarantees that the complement to the level curve of the specified polynomial
has a bounded connected component, which is the desired domain. This curve is
of degree $4$, since the polynomial $P(x,y)-1$ is irreducible when
$\alpha>0$. An admissible pair of polynomials $(F_1,F_2)$ for the
corresponding domain can be represented in the following form:
$F_1(z)=Cz^5-z$, $F_2(z_\beta)=C\gamma^5z_\beta^5-\gamma z_\beta$ for some
$C,\gamma\in\mathbb{C}$. Then the equality $F_1(z)=F_2(z_\beta)$ can be
rewritten as  $(z-\gamma z_\beta)(G(z,z_\beta)-1)=0$ for
$G(z,z_\beta)=C(z^5-\gamma^5z_\beta^5)\big/(z-\gamma z_\beta)$. Thus, if we
will show that $S_\beta P=G$, then the theorem will be proven.

Thus, our aim is to find the constants $C,\gamma,\alpha,\beta$ such that the
equality $S_\beta P=G$ holds. To do that we, firstly, find $C$ and $\gamma$
as functions of $\alpha$ and $\beta$, then we find $\alpha$ as a function of
$\beta$, and finally we find $\beta$. Observe, that adapting this method to
curves of higher order faces the problem that $\beta$ is found implicitly
as a solution of some equation, so in the case of curves of order greater
than 4, this method leads already to several equations on $\beta$, which have
no solution in the general case.

\section{Proofs.}

\begin{proof}[Proof of Lemma ~\ref{l1}.]
Let $P_1,P_2\in \mathbb{R}[x,y]$ and $Q=\gcd(P_1,P_2)$. The proof is based on the following fact (which can be found, for example, in
\cite[p.~16]{F08}): if $\deg Q=0$, then the intersection of the sets
$\varGamma_1=\{(x,y):P_1(x,y)=0\}$ and $\varGamma_2=\{(x,y):P_2(x,y)=0\}$ can only contain a finite set
of points. Therefore, if $\deg Q\leq n$, then the intersection of the sets
$\varGamma_1$ and $\varGamma_2$
cannot contain an algebraic curve of order greater than $n$. Since
$\deg\gcd(P_1/Q,P_2/Q)=0$, the intersection of $\varGamma_1$ and $\varGamma_2$
consists of the set $ \{(x,y):Q(x,y)=0\}$ and a finite set of points, therefore,
it cannot contain an algebraic curve of order higher than $n$.

Let a polynomial $R\in\mathbb{C}[x,y]$ of degree $n$ define an algebraic curve $\varGamma$
of order $n$. Since $R\in\mathbb{C}[x,y]$, then $R=R_1+iR_2$ for some
$R_1,R_2\in\mathbb{R}[x,y]$, and the curve $\varGamma$ is contained in the
intersection of the sets $\{(x,y):R_1(x,y)=0\}$ and $\{(x,y):R_2(x,y)=0\}$.
Thus, $\deg\gcd(R_1,R_2)=n$, $\deg R_1=\deg R_2=n$, $R_2=\delta R_1$ for some 
$\delta\in\mathbb R$, and, finally, $R=(1+i\delta)R_1$.

If the polynomial $P\in\mathbb{C}[x,y]$ of degree $m$ define the same curve
$\varGamma$, then $\deg\gcd(\gcd(P_1,P_2),R_1)=n$, where $P=P_1+iP_2$ for
some $P_1,P_2\in\mathbb{R}[x,y]$. Thus, we obtain that $P$ is divisible by
$R_1$.
\end{proof}

\begin{proof}[Proof of Lemma~\ref{l2}.]
Let $P\in\mathbb{C}[x,y]$, and
$\widetilde{P}=S_{\beta}\widetilde{S_{\beta}}P$. The polynomials $P$ and
$\widetilde{P}$ are equal as functions from $\mathbb{R}^2$ to $\mathbb{R}^2$,
therefore, the coefficients of $P$ and $\widetilde{P}$ coincide
\cite[p.~115]{Z64}, hence $S_\beta$ is invertible. Moreover, let $P=QT$,
where $Q,T\in\mathbb{C}[x,y]$, then the functions $P$ and
$\widetilde{S_{\beta}}(S_{\beta}Q S_{\beta}T)$ are similarly equal. Taking
into account the invertibility of $S_{\beta}$, we have
$S_{\beta}(QT)=S_{\beta}Q S_{\beta}T$, therefore, $S_{\beta}$ is an
automorphism.
\end{proof}

\begin{proof}[Proof of Theorem~\ref{th2}.]
Suppose the domain $D$ is $L$-special with an admissible pair $(F_1,F_2)$,
where $F_1$ and $F_2$ are polynomials of a complex variable, and assume that
$\max(\deg F_1,\deg F_2)\leq n$. Let $G(z,z_\beta)=F_1(z)-F_2(z_{\beta})$. If
$\max\deg(F_1,F_2)<n$, then $\widetilde{S_{\beta}}G$ defines $\varGamma$,
which contradicts to the fact that the order of $\varGamma$ equals to $n$.
Therefore, in what follows we will assume  that
$\max\deg(F_1,F_2)=\deg G=n$.

Let $P\in\mathbb{R}[x,y]$ define $\varGamma$ and $\deg P=n$. Then by
Lemma~\ref{l1} we have $\widetilde{S_{\beta}}G=\gamma P$ for some
$\gamma\in\mathbb{C}$. Define $\widetilde{G}=\gamma^{-1} G$, then
$\widetilde{S_{\beta}}\widetilde{G}=P$.

Consider the case when $n$ is even. From the definition of $G$ it follows 
that
$$
\widetilde{G}(z,z_{\beta})=
(\alpha_1+i\alpha_2)z^n+(\alpha_3+i\alpha_4)z_\beta^n+H(z,z_\beta),
$$
where $H$ has degree less than $n$. Since $P\in\mathbb{R}[x,y]$, the 
imaginary parts of all the coefficients of $P$ vanish. But
$\widetilde{S_{\beta}}\widetilde{G}=P$, whence,  by equating the
imaginary parts of the coefficients of $\widetilde{S_{\beta}} \widetilde{G}$
and $P$ at $x^n$, $x^{n-1}y$, $xy^{n-1}$, and $y^{n}$, we obtain 
\begin{equation*}
\alpha_2+\alpha_4=0,\enspace\alpha_2+\beta^{-n}\alpha_4=0,\enspace\alpha_1+\beta^{-1}\alpha_3=0,\enspace\alpha_1+\beta^{1-n}\alpha_3=0
\end{equation*}

From these equalities we have $\alpha_1=\alpha_2=\alpha_3=\alpha_4=0$, but
this contradicts to the fact that the degree of $G$ is equal to $n$. In the case
of odd $n$, the proof is similar.
\end{proof}

\begin{proof}[Proof of Theorem~\ref{th3}]
Let the polynomial $G(z,z_\beta)=F_1(z)-F_2(z_\beta)$ define $\varGamma$. If
$\widetilde{S}_\beta G=H(x,y)$, then, according to Lemma~\ref{l1},
$H(x,y)=(P(x,y)+C)R(x,y)$ for some $R\in\mathbb{C}[x,y]$ of degree $k$. In
view of Lemma~\ref{l2}, $S_\beta$ is an automorphism, therefore, $G=(S_\beta P+C)\,
S_\beta R$. For an arbitrary polynomial $Q\in\mathbb{C}[x,y]$ and arbitrary
positive integer $m\leq\deg Q$, we denote by $Q_{(m)}$ the sum of the
homogeneous monomials of degree $m$ containing in $Q$. Since $G_{(k)}$ is a
diagonal polynomial, and the degree of $R$ is $k<n$, then $S_\beta R_{(k)}$
is also a diagonal polynomial. On the other hand,
$$G_{(n+k)}=S_\beta P_{(n)}\, S_\beta R_{(k)},$$

Thus, taking into account Lemma~\ref{l3}, we arrive at a contradiction with
the fact that $n$ is not divisible by $k$.
\end{proof}

\begin{proof}[Proof of Theorem~\ref{th4}.]
Recall that the desired domain is the interior of the bounded connected
component of the complement to the curve defined by the equation $P(x,y)=1$
for $P(x,y)$ given by \eqref{pol}.

The polynomial $P(x,y)$ can be rewritten in the following form:
$$
P(x,y)=(x-(\alpha-i\alpha^{*})y)(x-(\alpha+i\alpha^{*})y)
(x-(-\alpha+i\alpha^{*})y)(x-(-\alpha-i\alpha^{*})y),
$$
where $\alpha^{*2}+\alpha^2=\beta^{-1}$ and $\alpha^{*}>0$ only depends  on
$\alpha$.

Let us check that the pair of polynomials $(F_1,F_2)$, where
$$F_1(z)=Cz^5-z,\quad
F_2(z_\beta)=C\gamma^5z_\beta^5-\gamma z_\beta $$
with some suitable $C$ and $\gamma$, can be taken as an admissible pair for
$D$. The equality $F_1(z)=F_2(z_\beta)$ is equivalent to the equality
$(z-\gamma z_\beta)(G(z,z_\beta)-1)=0$ with 
$G(z,z_\beta)=C(z^5-\gamma^5z_\beta^5)/(z-\gamma z_\beta)$.

Since
$$
G=C(z-e^{i\phi}\gamma z_\beta)(z-e^{2i\phi}\gamma z_\beta)(z-e^{3i\phi}\gamma z_\beta)
(z-e^{4i\phi}\gamma z_\beta),
$$
where $\phi=2\pi/5$, it is sufficient to find $\beta$, $\alpha$, $\gamma$,
$C_{p,q}$, $p,q=0,1$, such that
$$
S_{\beta}(x-((-1)^p\alpha+(-1)^q i\alpha^{*})y)=C_{p,q}(z-e^{i f(p,q)\phi}\gamma z_\beta),
$$
where
$
f(0,0)=2$, $f(0,1)=1$, $f(1,0)=3$, $f(1,1)=4.
$

Firstly, we find $C_{p,q}$ in the form of a suitable expressions of $\alpha$
and $\beta$. Direct calculations show that
\begin{multline*}
S_{\beta}((x-((-1)^p\alpha+(-1)^q i\alpha^{*})y))=\\=
(1-\beta)^{-1}\Bigl((1+(-1)^q\beta\alpha^{*}+(-1)^{p+1}i\beta\alpha)z-
(\beta+(-1)^q\beta\alpha^{*}+(-1)^{p+1}i\beta\alpha)z_\beta\Bigr),
\end{multline*}
whence
$C_{p,q}=(1-\beta)^{-1}(1+(-1)^q\beta\alpha^{*}+(-1)^{p+1}i\beta\alpha)$.
Thus,
$$
S_{\beta}((x-((-1)^p\alpha+(-1)^q i\alpha^{*})y))=C_{p,q}
(z-\gamma_{p,q}(\alpha,\beta)z_\beta)
$$
for
$$
\gamma_{p,q}(\alpha,\beta)=\frac{\beta+(-1)^q\beta\alpha^{*}+(-1)^{p+1}i\beta\alpha}
{1+(-1)^q\beta\alpha^{*}+(-1)^{p+1}i\beta\alpha}.
$$

Now we will find $\alpha$ (depending on $\beta$) satisfying both the
equalities
\begin{equation}\label{g00}
\gamma_{0,0}(\alpha,\beta)=e^{i\phi}\gamma_{0,1}(\alpha,\beta),
\end{equation}
\begin{equation}\label{g11}
\gamma_{1,1}(\alpha,\beta)=e^{i\phi}\gamma_{1,0}(\alpha,\beta).
\end{equation}
Taking into account the condition $\alpha^{*2}+\alpha^2=\beta^{-1}$ we see that the equation \eqref{g00}
is equivalent to
$$
(1-\beta)\alpha^{*}+i(-\alpha-\beta\alpha)=e^{i\phi}((\beta-1)\alpha^{*}+i(-\alpha-\beta\alpha)).
$$
Thus, $\alpha=\alpha(\beta)$ and $\alpha^{*}=\alpha^*(\beta)$ are connected by the relation
\begin{equation}\label{one}
\alpha^{*}(\beta)=\frac{(1+\beta)(1-\cos\phi)}{(1-\beta)\sin\phi}\,
\alpha(\beta).
\end{equation}
Moreover, for such $\alpha$ and $\alpha^{*}$ the equality \eqref{g11}
also holds. Note that
using~\eqref{one} and conditions $\alpha^{*2}+\alpha^2=\beta^{-1},\alpha>0$ and
$\alpha^{*}>0$, we can obtain an explicit formula, expressing $\alpha$ and
$\alpha^{*}$ in terms of $\beta$.

Next, we will find such $\beta_0$ that
$$
\gamma_{1,0}(\alpha(\beta_0),\beta_0)=e^{i\phi}\gamma_{0,0}(\alpha(\beta_0),\beta_0).
$$
Then we take $\gamma=e^{-i\phi}\gamma_{0,1}$ and the theorem is proven.

Let us prove the existence of $\beta_0$ such that
$$
\frac{\gamma_{1,0}(\alpha(\beta_0),\beta_0)}{\beta_0}=
e^{i\phi}\frac{\gamma_{0,0}(\alpha(\beta_0),\beta_0)}{\beta_0}
$$
Let $g_1(\beta)=\gamma_{1,0}(\alpha(\beta),\beta)/\beta$ and
$g_2(\beta)=e^{i\phi}\gamma_{0,0}(\alpha(\beta),\beta)/\beta$.

From the condition $\alpha^{*2}+\alpha^2=\beta^{-1}$  it follows that
$|g_1(\beta)|=|g_2(\beta)|=\sqrt{\beta}$. So, it suffices to find $\beta_0$
which is a root of the equation $\Re g_1(\beta)=\Re g_2(\beta)$ and, at the
same time, $\Im g_1(\beta_0)\Im g_2(\beta_0)>0$ holds.

The values of $\Re g_1$ and $\Re g_2$ can be directly calculated for $\beta=0.01$ and $\beta=0.1$. We have 
$$
\Re g_1(0.01)\approx 7.09,\enspace\Re g_2(0.01)\approx 8.89,\enspace
\Re g_1(0.1)\approx 2.83,\enspace \Re g_2(0.1)\approx 2.21.
$$

The functions $\alpha(\beta)$ and $\alpha^{*}(\beta)$ depend continuously on
$\beta$ for $0.01\leq\beta\leq0.1$; moreover, for such $\beta$ these
functions are positive. Therefore, $g_1(\beta)$ and $g_2(\beta)$ are also
continuous. Thus, from the intermediate value
theorem we deduce that there exists $0.01<\beta_0<0.1$  such that $\Re g_1(\beta_0)=\Re g_2(\beta_0)$.

On the other hand, for $0.01<\beta<0.1$ we have
\begin{align*}
\Im g_1(\beta)|g_1(\beta)|^2&=(1-\beta)\alpha(\beta)>0,\\
\Im g_2(\beta)|g_2(\beta)|^2&=
(2+(1+\beta)\alpha^{*}(\beta))\sin\phi+(\beta-1)\alpha(\beta)\cos\phi=\\
&=\Bigl (\frac{(1+\beta)^2}{1-\beta}(1-\cos\phi)+(\beta-1)\cos\phi\Bigr )
\alpha(\beta)+2\sin\phi>0.
\end{align*}
The theorem is proved.
\end{proof}

Next we will illustrate the construction given in the proof of Theorem~\ref{th4} by
some suitable picture. Direct computations  show that the values of $\beta$ and $\alpha$ found in the proof of Theorem~\ref{th4} are 
$\beta\approx0.039$, $\alpha\approx3.96$ and the curve $\varGamma$ is close
with respect to the Hausdorff metric to the curve $\widetilde{\varGamma}$
defined by the equation
\begin{equation*}
x^4+34.913x^2y^2+ 643.992y^4=1
\end{equation*}
and presented at the following picture.

\includegraphics[width=1\linewidth]{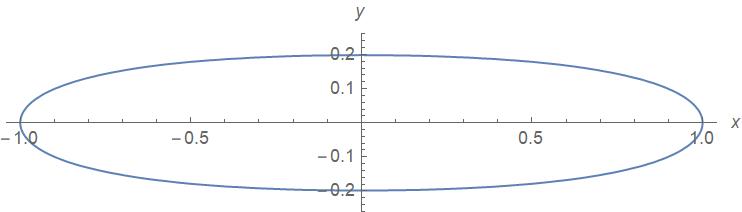}

Note that at the present moment the example constructed at Theorem~\ref{th4} is
unique, that is, the question of applying this method for the construction of
other $L$-special domains with fourth-order boundaries remains open.

\end{document}